\documentclass[11pt]{amsart}
\usepackage{amsmath}
\usepackage{bbm}
\usepackage{mathrsfs}
\usepackage[english]{babel}
\usepackage[utf8x]{inputenc}
\usepackage[T1]{fontenc}
\usepackage{wrapfig}

\usepackage{ dsfont }

\usepackage[top=3cm,bottom=2cm,left=3cm,right=3cm,marginparwidth=1.75cm]{geometry}

\usepackage{amsmath,amssymb}
\usepackage{graphicx}
\usepackage[colorinlistoftodos]{todonotes}
\usepackage[colorlinks=true, allcolors=blue]{hyperref}
\usepackage{cite}
\usepackage{enumitem} 
\usepackage[all,cmtip]{xy}
\usepackage{tikz-cd}
\usepackage{lmodern}
\usepackage{pgfplots}
\usetikzlibrary{intersections}
\usepackage{relsize}

\newtheorem{theorem}{Theorem}
\newtheorem{proposition}[theorem]{Proposition}
\newtheorem{corollary}[theorem]{Corollary}
\newtheorem{lemma}[theorem]{Lemma}
\newtheorem{definition}[theorem]{Definition}

\theoremstyle{remark} 
\newtheorem{remark}[theorem]{Remark}

\usepackage{mathtools} 
\mathtoolsset{showonlyrefs,showmanualtags}
\numberwithin{equation}{section}
\usepackage{amsmath}

\newcommand{\R}{\mathbb{R}}
\newcommand{\Ra}{\mathbb{R}_{\mathrm{an}}}
\newcommand{\Rae}{\mathbb{R}_{\mathrm{an}, \mathrm{exp}}}

\newcommand{\ZZ}{\mathscr{Z}}
\newcommand{\ZZo}{\mathscr{Z}}
\newcommand{\KK}{\mathscr{K}}
\newcommand{\KKo}{\mathscr{K}_0}
\newcommand{\dd}{\mathrm{d}}
\newcommand{\der}[1]{\frac{\dd}{\dd #1}}

\newcommand{\st}{\ |\ }

\newcommand{\set}[2]{\left\{ #1 \ | \ #2\right\}}

\DeclareMathOperator{\supp}{supp}
\DeclareMathOperator{\Vol}{vol}
\DeclareMathOperator{\Span}{Span}

\newcommand{\be}{\begin{equation}}
\newcommand{\ee}{\end{equation}}

\title{On tameness of zonoids}

\author{Antonio Lerario}
\address{SISSA, Via Bonomea 265, 34136 Trieste, Italy}
\email{lerario@sissa.it}
\author{L\'eo Mathis}
\address{SISSA, Via Bonomea 265, 34136 Trieste, Italy}
\email{leo.mathis@sissa.it}

\begin{document}

\maketitle
\begin{abstract}
 We prove that in a globally subanalytic family of convex bodies the set of zonoids is log-analytic, and in particular it is definable in the o--minimal structure generated by globally subanalytic sets and the graph of the exponential function.
\end{abstract}

\section{Introduction}
A \emph{zonoid} in $\R^n$ is a convex body that can be approximated, in the Hausdorff metric, by finite sums of line segments. Zonoids appear in several different contexts of modern mathematics: functional analysis, measure theory, stochastic geometry, geometric convexity, enumerative geometry \cite{A, Bolker, BLM, PSC, goodeyweil, GWbook, PSC2, NRZ, Vitale, Wi}. Despite the simple definition, the problem of deciding whether a given convex body is a zonoid is considerably hard, and it is referred to as the zonoid problem. 

In $\R^2$ every centrally symmetric convex body is a zonoid, but for $n\geq 3$ no  simple geometric characterization exists for zonoids in $\R^n$. For example, an interesting non-trivial characterization says that a convex body is a zonoid if and only if its support function is conditionally positive definite\footnote{A function $f:\R^n\to \R$ is said to be \emph{conditionally positive definite}, see \cite[Notes for Section 3.5, pag. 204]{bible}, if for all $k\in \mathbb{N}$ and for all $x_1, \ldots, x_k\in \R^n$ and $\alpha_1, \ldots, \alpha_k\in \R$ such that $\alpha_1+\cdots+\alpha_k=0$ we have
$ \sum_{i,j=1}^k\alpha_i\alpha_jf(x_i-x_j)\geq 0.$
%A convex body $K\subset \R^n$ is a zonoid if and only if its support function $h_K:\R^n\to \R$ is conditionally positive definite.
}: this characterization involves nice inequalities on the values of the support function, but an \emph{infinite} number of them. 

In this paper we show that, as soon as we restrict to a certain ``tame'' family of convex bodies, the set of zonoids in this family can be described using only a finite number of ``tame'' conditions. Here we will express the tameness condition in terms of definability in some o--minimal structure. The structures that we will consider are the o--minimal structure generated by globally subanalytic sets, denoted by $\Ra$, and the o--minimal structure generated by $\Ra$ and the graph of the exponential function, denoted by $\R_{\mathrm{an}, \mathrm{exp}}$.
%For $n\geq 3$ no  simple geometric characterization exists for zonoids in $\R^n$. An interesting non-trivial characterization can be given in terms of support functions. A function $f:\R^n\to \R$ is said to be \emph{conditionally positive definite} if for all $k\in \mathbb{N}$ and for all $x_1, \ldots, x_k\in \R^n$ and $\alpha_1, \ldots, \alpha_k\in \R$ such that $\alpha_1+\cdots+\alpha_k=0$ we have
%\be\label{eq:pos} \sum_{i,j=1}^k\alpha_i\alpha_jf(x_i-x_j)\geq 0.\ee
%A convex body $K\subset \R^n$ is a zonoid if and only if its support function $h_K:\R^n\to \R$ is conditionally positive definite. This characterization shows that, in general, the zonoid problem requires testing an infinite number of inequalities. Assuming that the support function $h_K$ is semialgebraic, then for every $k\in \mathbb{N}$ the condition \eqref{eq:pos} is semialgebraic. However, even in this restricted situation, for deciding whether $K$ is a zonoid, \eqref{eq:pos} should be verified for \emph{infinitely many $k$}, and therefore this characterization is not given by a semialgebraic condition.
\begin{theorem}\label{thm:main}
	If a family of convex bodies is definable in $\Ra$, the set of zonoids in this family is definable in $\Rae$.
\end{theorem}

\begin{remark}Our study was motivated by the following question, suggested by B. Sturmfels: given a semialgebraic convex body, is there an algorithm to decide whether it is a zonoid? For example, let $p\in\R[x_1,\ldots,x_n]_d$ be a polynomial of degree $d$ in $n$ variables and let $K_p:=\set{x\in\R^{n}}{p(x)\geq0}$. We define $P:=\set{p\in\R[x_0,\ldots,x_n]_d}{\textrm{$K_p$ is a convex body}}$. This defines a tame (semi-algebraic) family of convex bodies. Is there an algorithm that, given the polynomial $p$, decides whether $K_p$ is a zonoid? To answer this question, consider the set: 
\begin{equation}
    \ZZ(P):=\set{p\in P}{K_p\text{ is a zonoid}}.
\end{equation}
If $\ZZ(P)$ is semialgebraic, such an algorithm (in the sense of \cite{BPR}) exists. In general, for $n\geq 3$,  it is not known if $\ZZ(P)$ is semi-algebraic. Our Theorem \ref{thm:main} above  proves that it is at least definable in the o--minimal structure $\R_{\mathrm{an,exp}}$, and therefore it cannot be ``wild'' (in particular it shares with semialgebraic sets many finiteness properties). Our proof of Theorem~\ref{thm:main} is constructive but it is not clear to us whether this can be called an algorithm.
\end{remark}
\begin{remark}
If the support function -- and not just the body -- is polynomial, then it is easy to see that the condition of being a zonoid is semialgebraic in the coefficients of the polynomial.
\end{remark}

\begin{remark}
As we will see below, a crucial ingredient in our proof is the fact that a parametric integral of a function which is definable in $\Ra$ is definable in $\Rae$. A major open question in o--minimal geometry is the following: given a function $f:P\times \R^m\to \R$ definable in some o--minimal structure $\mathcal{M}$, is there an o--minimal expansion $\mathcal{M}'$ such that the function $p\mapsto \int_{\R^m}f(p,x)\mathrm{d}x$ is definable in $\mathcal{M}'$? This is true for $\mathcal{M}=\Ra$, for which we can take $\mathcal{M}'=\Rae$. If this question has an affirmative answer, then the proof of Theorem \ref{thm:main} can be easily extended to get the following statement: given a family of convex bodies definable in some o--minimal strucutre $\mathcal{M}$, the set of zonoids in this family is definable in some o--minimal expansion $\mathcal{M}'$. For the moment this is just a conjecture.
\end{remark}
\begin{remark}
The main idea of the proof is to reduce the problem to the study of a ``tame'' family of distributions on the unit interval. More precisely, our first step consists in associating to our family of convex bodies a new tame family of rotational invariant convex bodies. The support functions of these convex bodies depend on one variable only and zonoids correspond to members of the family such that an appropriate integral transform of their support function is a non--negative measure. Using the language of distributions, we show that the condition of being a non--negative measure can be expressed in a tame way using a second primitive of the integral transform. 
\end{remark}

\subsection{Acknowledgements}The authors wish to thank B. Sturmfels for suggesting the question from which this paper started, and S. Basu for many helpful comments. The authors are also indebted to T. Kaiser, who pointed out an error in the formulation of the main result in the first version of this paper and  helped us making the presentation more clear.
\section{Preliminaries}

\subsection{Zonoids}Let us recall some definitions from convex geometry. For a more detailed treatment we refer to \cite{bible}. Our main object of study are 
\emph{convex bodies}, i.e. non-empty, compact and convex subsets of Euclidean space. Recall that the \emph{support function} of a convex body $K\subset \R^{n+1}$ is the function on $h_K:S^n\to \R$ given by:
	\begin{equation}
		h_K(u):=\sup \left\{ \langle u, x \rangle \st x\in K \right\}.
	\end{equation}
A convex body $K$ is called a \emph{zonoid} if there exists an even, non--negative measure $\mu$ on $S^n$ such that the support function of $K$ can be written in the following form
	\begin{equation}
	\label{eq:defzon}	h_K(u)=\frac{1}{2}\int_{S^n} |\langle u, x \rangle| \dd\mu(x)
	\end{equation}
%The set of convex bodies in $\R^{n+1}$ will be denoted $\KK^{n+1}$. The set of convex bodies which are also symmetric with respect to the origin will be denoted by $\KKo^{n+1}$, while $\ZZo^{n+1}$ will be the set of zonoids in $\R^{n+1}$. 

We will use the following notion introduced by Lonke in \cite{lonke}.
\begin{definition}\label{def:espin}
	Let $f:S^n\to \R$ be a measurable function and $e\in S^n$. We denote by \be O(e):=\mathrm{Stab}_{O(n+1)} (e)\simeq O(n)\subset O(n+1)\ee the stabilizer of $e$ in $O(n+1)$, endowed with the normalized Haar measure $\mathrm{d}g$. We then define the measurable function $S_ef:S^n\to \R$ by
	\be 
		S_ef(u):=\int_{O(e)}f\left(g(u)\right) \mathrm{d}g.
	\ee
	If $K\subset \R^{n+1}$ is a centrally symmetric convex body whose support function is $h_K$ we define $S_eK$ to be the convex body whose support function is given by:
	\begin{equation}
		h_{S_eK}:=S_eh_K.
	\end{equation}
\end{definition}

	The fact that $S_eh_K$ is the support function of a convex body follows from the characterization of support functions as positively homogeneous sublinear functions~\cite[Section~1.7]{bible}.
	
Using Definition \ref{def:espin} one can give the following alternative characterization of zonoids, see \cite[Theorem 1.1]{lonke}.
\begin{lemma}\label{lem:Lonke}
 	A convex body $K\subset \R^{n+1}$ is a zonoid if and only if for all $e \in S^n$, $S_eK$ is a zonoid.
 \end{lemma}

\subsection{o--minimal structures and tameness}\label{sec:definable}

%\commleo{I took the definition of  \cite{vandendries}. Maybe we will remove it later.}
	We denote by $\mathbb{R}_{\mathrm{an}}$ the o--minimal structure generated by the globally subanalytic sets and by $\mathbb{R}_{\mathrm{an}, \mathrm{exp}}$ the o--minimal structure generated by the globally subanalytic sets together with the graph of the exponential function, see \cite{vandendries}. 
	
	If $P$ is a globally subanalytic set, following \cite{CluckersMiller}, we denote by $\mathscr{C}(P)$ the $\R$-algebra of real valued functions generated by all globally subanalytic functions on $X$ and all the functions of the form $x\mapsto \log f(x)$, where $f:X\to (0, \infty)$ is globally subanalytic. A function in $\mathscr{C}(X)$ is called a \emph{constructible} function. Notice that functions definable in $\mathbb{R}_{\mathrm{an}}$ are constructible and that constructible functions are definable in $\mathbb{R}_{\mathrm{an}, \mathrm{exp}}.$
	
	In the sequel we will simply say that a set or a function definable in $\Ra$ is \emph{subanalytic} (omitting the word ``global'').
	
%	In the sequel we will say that a set or a function is $\mathbb{R}_{\mathrm{an}}$--tame (respectively $\mathbb{R}_{\mathrm{an}, \mathrm{exp}}$--tame) if it is definable in $\mathbb{R}_{\mathrm{an}}$ (respectively in $\mathbb{R}_{\mathrm{an}, \mathrm{exp}}$).
	
	%We say that a set or a function is \emph{tame} if it is definable in some o--minimal structure, in our case  either in $\mathbb{R}_{\mathrm{an}}$ or in $\mathbb{R}_{\mathrm{an}, \mathrm{exp}}$.

 We will use the following crucial result \cite[Theorem 1.3]{CluckersMiller}.

\begin{theorem}\label{thm:stabbyint}
	 Let $P$ be subanalytic and $F\in \mathscr{C}(P\times \mathbb{R}^m)$. Suppose that for all $p\in P$ the function $F(p,\cdot):\R^m\to\R$ is integrable. Then the the function $I(F):P\to \R$ defined by $ I(F)(p) := \int_{\R^m}F(p, x)\mathrm{d}x$ is constructible, and in particular definable in $\Rae$.
\end{theorem}

\begin{remark}
	In the case $F:P\times \mathbb{R}^m\to \R^m$ is semialgebraic, then the parametrized integral function $I(F)$ is definable in a structure strictly smaller than $\R_{\mathrm{an}, \mathrm{exp}}$, see \cite{kaiser}.
\end{remark}

\begin{corollary}\label{cor:tamespin}
	If $h:P\times S^n\to \R$ is constructible, the function $(e,p, u)\mapsto S_eh(p,u)$ is also constructible.

	\begin{proof}
		Consider a subanalytic function $F:S^n\times O(n)\to O(n+1)$ such that for almost all $e\in S^n$ the function $F(e,\cdot)$ is a subanalytic isomorphism between $O(n)$ and $O(e)$. (Since we are only requiring that $F$ is definable, such function can also be defined piecewise.) Then we can write:
		\be S_eh(p,u)=\int_{O(n)} h(p,\left(F(e,\tilde{g})(u)\right) |\det JF(e,\cdot)| d\tilde{g},\ee
		where $d\tilde{g}$ is the normalized Haar measure on $O(n)$. Since the integrand is constructible, the result follows by applying Theorem \ref{thm:stabbyint} after noticing that there is a diffeomorphism, definable in $\Ra$, between an open dense of subset of $O(n)$ and $\R^m$, $m=n(n-1)/2$ (for instance one can take the restriction of the Riemannian exponential map at the identity on an appropriate subanalytic domain). 
	\end{proof} 
\end{corollary}

\begin{definition}\label{def:tamefamily}
	Let $P$ be a subanalytic set. A subanalytic family of convex bodies in $\R^{n+1}$ is a subanalytic set $T\subset P\times \R^{n+1}$ such that for every $p\in P$ the set 
	\be K_p:= \set{x\in \R^{n+1}}{(p,x)\in T}\ee
	is a convex body. For $p\in P$, we will denote by $h_p$ the support function of $K_p$ (instead of $h_{K_p}$).
\end{definition}
If $T\subset P\times \R^{n+1}$ is a subanalytic family of convex bodies, the function  $H:P\times S^n\to \R$, given by $H(p, u)=h_{K_p}(u)$, is subanalytic. 
Moreover, if $P$ and $T\subseteq P\times \R^{n+1}$  are subanaytic, denoting  by $T_p:=\set{x\in \R^{n+1}}{(p,x)\in T},$ it is immediate to see that the following sets are subanalytic:
\begin{enumerate}
\item[(i)] $\KK(P):=\set{p\in P}{T_p\text{ is a convex body}};$
\item[(ii)] $\KKo(P):=\set{p\in P}{T_p\text{ is a centrally symmetric convex body, centered at the origin}}.$
\end{enumerate}
The following lemma is a variation of the classical Cell Decomposition Theorem from \cite{vandendries}, adapted to our needs.
\begin{lemma}\label{lem:tamnonsmooth}Let $\mathcal{F}=\{f_1, \ldots, f_\rho\}$ be a finite family of real valued, subanalytic functions on $S\times [-1,1]$, where $S$ is a subanalytic set. For every $m\in \mathbb{N}$ there exists a partition of $S$ into finitely many subanalytic sets $S=\coprod_{j=1}^s S_j$ and for every $j=1, \ldots, s$ there are subanalytic functions 
\be -1\equiv a_{j, 1}<a_{j,2}< \cdots< a_{j, \nu_{j}-1}<a_{j, \nu_{j}}\equiv 1:S_j\to [-1,1]\ee
 such that for every $r=1, \ldots, \rho$, for every $i=1, \ldots, \nu_j-1$ and for every $x\in S_j$ the function 
\be f_r(x, \cdot)|_{(a_{j,i}(x), a_{j, i+1}(x))}:(a_{j,i}(x), a_{j, i+1}(x))\to \R\ee
is of class $C^{m+2}.$
\end{lemma}
\begin{proof}Consider the subanalytic set $B\subset P\times (-1,1)$ consisting of all the pairs $(x, z)$ such that for all $r=1, \ldots, \rho$ the function $f_r(x, \cdot):(-1, 1)\to \R$ is of class $C^{m+2}$ at $z\in (-1,1)$. The set $B$ is subanalytic because it is defined by a first order formula in  $\Ra$. Let $\Sigma$ denote the complement of $B$ in $S\times (-1,1)$ and denote by $p_1:S\times(-1,1)\to S$ the projection on the first factor. 

By the Trivialization Theorem \cite[Theorem (1.7), Chapter 9]{vandendries} there exists a partition $S=\coprod_{j=1}^sS_j$ into subanalytic sets such that map $p_1:S\times (-1,1)\to S$ can be trivialized over each $S_j$, respecting $\Sigma$. This means that there are subanalytic sets $T_1, \ldots, T_s\subset (-1,1)$ and subanalytic homeomorphisms 
\be \psi_j:S_j\times (-1,1)\to S_j\times (-1,1)\ee
 such that $p_1(\psi_j(x,z))=x$ for every $x\in S_j$ and $\psi_j(S_j\times T_j)=(S_j\times (-1,1))\cap \Sigma$.

Observe that for every $x\in S$ the set $p_1^{-1}(x)\cap \Sigma$ is a zero-dimensional subanalytic set, since the function $f(x, \cdot)$ is subanalytic and therefore of class $C^{m+2}$ outside of a finite number of points. Therefore each $T_j$ is zero-dimensional and (if nonempty) consists of a finite set of points, which we assume to be ordered:
\be T_j=\{t_{j, 2},\ldots,t_{j, \nu_j-1}\}.\ee

Denote now by $p_2:S\times (-1,1)$ the projection on the second factor and pick $j\in \{1, \ldots ,s\}$. If $T_j=\emptyset$, then we set $\nu_j=2$, $a_{j,1}(x)\equiv -1$ and $a_{j,2}(x)\equiv 1.$ If $T_j\neq \emptyset$ we define, for $i=2, \ldots, \nu_j-1$, the desired functions by:
\be a_{j, i}(x)=p_2(\psi_j^{-1}(x, t_{j,i})).\ee

\end{proof}
%\begin{lemma}\label{lem:tamnonsmooth}
%Let $F:P\times [-1, 1]\to \R$ be a tame function. For every $p\in P$ and for $k\in \mathbb{N}$ consider the set $\Sigma_k(p)\subset [-1, 1]$ defined by $\Sigma_k(p):=\set{z\in[-1,1]}{F(p,\cdot) \text{ is not $C^k$ at $z$}}$. Then the set $
% 	\Sigma_k:=\set{(p, z)\in P\times [-1, 1]}{ z\in \Sigma_k(p)}$ is tame.
%\end{lemma}
%\begin{proof}
%Consider the case $k=1$. The function $F(p,\cdot)$ is $C^1$ at $z$ if and only if the limit $\lim_{h\to0}\frac{F(p,z+h)-F(p,z)}{h}$ exists and is finite. This is a first order formula involving only tame objects. This proves that the complement of $\Sigma_1$ is tame and thus $\Sigma_1$ itself is tame. The general case is proved similarly.
%\end{proof}
\subsection{Distributions}
	Let $X$ be a smooth manifold. We denote by $C^\infty_c(X)$ the space of compactly supported, smooth functions on $X$, endowed with the $C^{\infty}$ topology (see \cite{Hirsch}). The space of \emph{distributions} on $X$ is denoted by $\mathcal{D}'(X)$: it is  the dual of $C^\infty_c(X)$, endowed with the weak--$*$ topology.
	
\begin{remark}
	If $X$ is compact $C^\infty_c(X)=C^\infty(X)$.
\end{remark}

\begin{remark}
	Let $G$ be a group. Any group action $G\curvearrowright X$ gives rise to a group action on $C^\infty_c(X)$ by $g\cdot\varphi(x):=\varphi\left( g^{-1}\cdot x\right) $ for all $g \in G$, $\varphi\in C^\infty_c(X)$ and $x \in X$. This induces a group action on $\mathcal{D}'(X)$ defined by $\langle g\cdot\rho, \varphi \rangle :=\langle\rho, g^{-1}\cdot \varphi \rangle$ for all $g \in G$, $\varphi\in C^\infty_c(X)$ and $\rho \in \mathcal{D}'(X)$. We denote by $C^\infty_c(X)^G$ (respectively $\mathcal{D}'(X)^G$) the functions (respectively distributions) which are invariant under these actions. Note that $\left(C^\infty_c(X)^G\right)^*=\mathcal{D}'(X)^G$.
\end{remark}

In our case $X$ will be $S^n$, $[-1,1]$ or $(-1,1)$ an $G$ will be $O(e)$ for some $e\in S^n$.

\begin{definition}
	We denote by $C^\infty_{even}(S^n)^{O(e)}$ (respectively $\mathcal{D}'_{even}(S^n)^{O(e)}$) the functions (respectively distributions) on $S^n$ invariant by the group generated by $O(e)$ and the antipodal map $u\mapsto -u$. Similarly we denote by $C^\infty_{even}([-1,1])$ (respectively $\mathcal{D}'_{even}([-1,1])$) the even functions (respectively distributions) on $[-1,1]$.
\end{definition}

\begin{remark}\label{rk:smoothinDist}
	Recall that there is a dense embedding $C^\infty_{even}(S^n)\hookrightarrow\mathcal{D}'_{even}(S^n)$, $f\mapsto \rho_f$ given, for all $g\in C^\infty_{even}(S^n)$, by:
	\begin{equation}
		\langle \rho_f, g \rangle := \int_{S^n} f(x) g(x) \dd x,
	\end{equation}
	where  $\dd x$ denotes the integration with respect to the standard volume form on $S^n$. There is a similar dense embedding for $[-1,1]$.
\end{remark}

\begin{lemma}\label{lem:P}
	Consider the map $\sigma: [-1,1]\to [0,1]$ given by $\sigma(z):=\sqrt{1-z^2}$. Then the operator $\mathcal{P}$ defined for all $\varphi\in C^\infty_{even}([-1,1])$ by $\mathcal{P}\varphi:=\varphi\circ \sigma$ is a continuous automorphism of $ C^\infty_{even}([-1,1])$.
	
	\begin{proof}
		It is enough to know that for every $\varphi\in C^{\infty}_{even}([-1, 1])$ there exists $\psi\in C^\infty([0,1])$ such that $\varphi(z)=\psi(z^2)$ and that the map $\varphi\mapsto \psi $ is continuous in the $C^\infty$ topology. This is done in~\cite{whitney1943}.
	\end{proof}
\end{lemma}

\begin{proposition}\label{prop:zetan}
	For every $e\in S^n$ there is a linear homeomorphism 
	\be \zeta_n: C^\infty_{even}(S^n)^{O(e)}\to  C^\infty_{even}([-1,1]).\ee Its inverse is given by $(\zeta_n^{-1}\varphi)(u):=\varphi\left(\langle u, e \rangle\right)$. By density this gives a linear homeomorphism $		\zeta_n:\mathcal{D}'_{even}(S^n)^{O(e)}\to \mathcal{D}'_{even}([-1,1]).$	\begin{proof}Let $f\in C^\infty_{even}(S^n)^{O(e)}$ and let us consider an orthonormal basis $v_1,\ldots,v_n,e$. In this basis \be (\zeta_nf)(z)=f(0,\ldots,0,\sqrt{1-z^2},z).\ee
		Denoting by $\gamma: [-1,1]\to S^n$ the curve $\gamma(z):=(0,\ldots, 0 ,\sqrt{1-z^2},z)$, we see that  $(\zeta_nf)=f\circ \gamma$. 
		
		This function is cleary even and $C^{\infty}$ away from $z=\pm 1$. We show now that it is also smooth near these points. Consider the local coordinate chart $\pi$ around the north pole $(0,\ldots,0,1)$ that is the projection on the first $n$ coordinates. Then near the north pole, $\zeta_nf=\tilde{f}\circ \pi \circ \gamma$ where $\tilde{f}:=f\circ\pi^{-1}$. The function $\tilde f$ is smooth because it is a composition of smooth functions. Consider now first the case $n=1$, for which we have \be\label{eq:cc} (\zeta_1f)(z)=\tilde{f}(\sqrt{1-z^2})=(\mathcal{P}\tilde{f})(z).\ee In this case it is enough to apply the previous lemma. For the general case we note that $\zeta_n$ is the composition of $\zeta_1$ with the restriction to any $2$-plane containing $e$. This proves that $\zeta_n f\in C^\infty_{even}([-1,1]).$ 
		
		The map $\zeta_n$ is linear and, by \eqref{eq:cc} and Lemma \ref{lem:P}, it is also continuous. Its inverse, given by $(\zeta_n^{-1}\varphi)(u):=\varphi\left(\langle u, e \rangle\right)$, is also clearly continuous in the $C^{\infty}$ topology.
	\end{proof}
\end{proposition}

The isomorphism $\zeta_n=\zeta_{n, e}$ defined above depends on $e\in S^n$. However this dependence is constructible, in the following sense.

\begin{proposition}\label{prop:zetatame}
	Let $F:P\times S^n\times S^n\to \R$ be a constructible function such that for all $e\in S^n$  and $p\in P$ the function $f_{p,e}:=F(p,e,\cdot)$ is $O(e)$-invariant. Then the function $(p,e,z) \mapsto (\zeta_{n,e}f_{p,e})(z)$ is a constructible function on $P\times S^n\times [-1,1]$.
	\begin{proof}
		Take a subanalytic function $v_0:S^n\to S^n$ such that for all $ e\in S^n$, $\langle e, v_0(e)\rangle=0$. Then we can write $(\zeta_{n,e}f_{p,e})(z)=F(p,e,ze+\sqrt{1-z^2}\ v_0(e))$.  Since $F$ is constructible, then:
		\be F(p,x,y)=\sum_{i=1}^a\left(g_i(p,x,y)\prod_{j=1}^{b_i}\log f_{i,j}(p,x,y)\right),\ee
		for some subanalytic functions $g_{i}:P\times S^n\times S^n\to \R$ and $f_{i,j}:P\times S^n\times S^n\to (0, \infty).$
		
		In particular:
		\begin{align} (\zeta_{n,e}f_{p,e})(z)&=\sum_{i=1}^a\left(g_i(p,e,ze+\sqrt{1-z^2}\ v_0(e))\prod_{j=1}^{b_i}\log f_{i,j}(p,e,ze+\sqrt{1-z^2}\ v_0(e))\right)\\
		&=\sum_{i=1}^a\left(\tilde{g}_i(p,e,z)\prod_{j=1}^{b_i}\log \tilde{f}_{i,j}(p,e,z)\right)
			\end{align}
		with $\tilde{g}_{i,j}:P\times S^n\times [-1,1]\to \R$ and $\tilde{f}_{i,j}:P\times S^n\times [-1,1]\to (0, \infty)$ defined by $\tilde{g}_i(p,e,z)=g_i(p,e,ze+\sqrt{1-z^2}\ v_0(e))$ and $\tilde{f}_{i,j}(p,e,z)=f_{i,j}(p, e,ze+\sqrt{1-z^2}\ v_0(e)).$ Since $\tilde{g}_i$ and $\tilde{f}_{i,j}$ are subanalytic, then  $(p,e,z) \mapsto (\zeta_{n,e}f_{p,e})(z)$ is constructible.
	\end{proof}
\end{proposition}

\begin{definition}
 Let $\varphi$ be a function on $[-1,1]$ or $(-1,1)$. Suppose $\varphi$ is locally integrable and thus defines a distribution. We will denote by $D^k \varphi$ its $k$-th distributional derivative and by $\varphi^{(k)}$ the classical derivative whenever this is defined and locally integrable (we will also use $\frac{\dd^k \varphi}{\dd z^k}$ when we need to specify the variable of differentiation), i.e.:
 \be \langle D^k\varphi, \psi\rangle=(-1)^k\int_{I}\varphi\, \psi^{(k)}\quad \mathrm{and}\quad  \langle \varphi^{(k)}, \psi\rangle=\int_{I}\varphi^{(k)}\, \psi.\ee
\end{definition}
	Let $\rho\in \mathcal{D}'(X)$. We say that $\rho$ is \emph{non negative} and write $\rho \geq 0$ if for all $\varphi \in C^\infty_c(X)$ such that $\varphi\geq 0$ we have $\langle \rho, \varphi \rangle\geq 0$.

Let us recall some properties of distributions. For proof and details, see \cite{Schwartz}.
For $x\in X$ we denote by $\delta_x$ the Dirac delta measure at $x$.

\begin{proposition} \label{prop:propdist} Distributions satisfy the following properties:
	\begin{itemize}
		\item[(i)]let $\rho\in \mathcal{D}'(X)$. Then $\rho \geq 0$ if and only if $\rho$ is a positive (finite) Borel measure on $X$;	
		\item[(ii)]let $x\in [-1,1]$ and let $\rho\in\mathcal{D}'\left([-1,1]\right)$ such that $\supp(\rho)=\{x\}$. Then there exist $N\in \mathbb{N}$ and $r_0, \ldots, r_N$ such that $\rho=\sum_{i=0}^N r_i D^i\delta_x$; if moreover $\rho\geq 0$, then $r_0\geq 0$ and $r_i=0$ for all $i>0;$
		\item[(iii)]let $\rho\in\mathcal{D}'\left((-1,1)\right).$ Then $\rho\geq 0$ iff there is a convex function $\gamma : (-1,1) \to \R$ such that $D^2\gamma=\rho$.
	\end{itemize}
\end{proposition}

\begin{remark} \label{rk:posdist}
	Let $\rho \in \mathcal{D}'(S^n)^{O(e)} $. Then $\rho\geq 0$ in $\mathcal{D}'(S^n)^{O(e)}$ if and only if $\zeta_n\rho\geq 0$ in $\mathcal{D}'_{even}([-1,1])$. This is because $f \geq 0$ iff $\zeta_nf \geq 0$.
\end{remark}

 A straightforward computation gives also the following lemma.
\begin{lemma}\label{lem:compdelt1} Let $\psi$ be $k$ times differentiable on $[a,b]$ with $-1\leq a<b\leq1$. Then the following properties are true:
	\begin{itemize}
		\item[(i)] If $\psi^{(k)}\in L^1([a,b])$, then 
		\be D^k \psi = \psi^{(k)}+\sum_{i=0}^{k-1}\psi^{(i)}(a)D^{k-1-i}\delta_a-\psi^{(i)}(b)D^{k-1-i}\delta_b;\ee
		\item[(ii)]for every $x\in[a,b]$ we have \be \psi\cdot D^{k}\delta_x=(-1)^k\sum_{i=0}^k(-1)^i \binom{k}{i} \psi^{(k-i)}(x) D^i\delta_x.\ee
\end{itemize}
\end{lemma}
Next lemma characterizes non--negative measures on $[-1, 1]$.
\begin{lemma}\label{lem:SchwartzOnClosed}
	Let $\Lambda \in \mathcal{D}'([-1,1])$ and let  $W^{2,1}\left([-1,1]\right)$ be the Sobolev space of (equivalence classes of) functions on $[-1,1]$ with integrable second derivative. Then $\Lambda\geq0$ if and only if the following two conditions are both satisfied:
	\begin{itemize}
		\item[(i)] There is $\gamma: [-1,1]\to \R$ in $W^{2,1}([-1,1])$ convex such that $D^2\gamma|_{(-1,1)}=\Lambda|_{(-1,1)};$
		\item[(ii)] There exist $\lambda_{-1}, \lambda_{1}\geq 0$ such that: 
	\be\label{eq:L-D2g}
		\Lambda-\gamma^{(2)}=\lambda_{-1}\delta_{-1}+\lambda_{1}\delta_{1}.\ee
	\end{itemize}
	
	\begin{proof}
	Assume that (i) and (ii) are satisfied. Then $\gamma^{(2)}\geq 0$ and therefore \eqref{eq:L-D2g} defines a positive measure.
	
		%First of all if $(i)$ is satisfied then $\Lambda-D^2\gamma$ has support on $\left\{\pm 1\right\}$. Applying Proposition \ref{prop:propdist} we get that $\Lambda-D^2\gamma$ is of the form~\eqref{eq:L-D2g} and thus $(ii)$ makes sense.
		
		%If $(i)$ and $(ii)$ are satisfied, it is pretty clear that $\Lambda$ is a non negative measure.
		
		Now suppose $\Lambda\geq 0$. Then $\Lambda|_{(-1,1)} \geq 0$ and by point (iii) of Proposition \ref{prop:propdist} there is $\gamma:(-1,1)\to\R$ convex such that $D^2\gamma|_{(-1,1)}=\Lambda|_{(-1,1)}$. Since $\Lambda$ is finite on the closed interval $[-1,1]$, then $\gamma^{(2)}$ is integrable near $\{\pm1\}$, i.e. $\gamma$ is in  $W^{2,1}\left([-1,1]\right)$ (in particular $\gamma$ is also defined in $\{\pm1\}$). This shows that (i) is satisfied.
		
		Now since $\gamma^{(2)}=D^2\gamma|_{(-1,1)}=\Lambda|_{(-1,1)}$, then $\Lambda-\gamma^{(2)}$ has support on $\{\pm 1\}$ and by point (ii) of Proposition \ref{prop:propdist} it is a sum of deltas and its derivatives. Since both $\Lambda$ and $\gamma^{(2)}$ are measures, then $\Lambda-\gamma^{(2)}$ has no derivatives of deltas. Thus there exist $\lambda_{1}, \lambda_{-1}\in \R$ such that $\Lambda=\gamma^{(2)}+\lambda_{-1}\delta_{-1}+\lambda_{1}\delta_{1}$.  Since $\Lambda\geq 0$  point (ii) follows.
		\end{proof}
\end{lemma}
\subsection{Integral Transforms}
We recall now some classical construction of integral transforms. We refer the reader to \cite{goodeyweil} for more details.

The \emph{Cosine Transform} $T_n$ and the \emph{Radon Transform} $R_n$ are the endomorphisms of $C^\infty_{even}(S^n)$ given for all $f\in C^\infty_{even}(S^n)$ by
	\begin{equation}
		T_nf(u):=\frac{1}{2}\int_{S^n}|\langle u , x \rangle| f(x) \dd x\quad \textrm{and}\quad 	R_nf(u):=\int_{S\left(u^\perp\right)}f(x) \dd x,
	\end{equation}
	where $S\left(u^\perp\right)=S^n\cap u^\perp$ is the unit sphere on $u^\perp$ and $ \dd x $ denotes the integration with respect to the standard volume forms of the corresponding spheres.

\begin{remark}\label{rk:prolTR}
	The two operators $T_n$ and $R_n$ are linear continuous bijections from the space $C^\infty_{even}(S^n)$ to itself. One can define the transpose of these operators on the dual space $\mathcal{D}'_{even}(S^n)$ in the usual way. More precisely, for every $\rho\in \mathcal{D}'_{even}(S^n)$ and for every $f \in C^\infty_{even}(S^n)$, we set:
	\be\langle (T_n)^t \rho, f \rangle := \langle \rho, T_n f\rangle,\ee and similarly we do for $(R_n)^t$. 
	
	With this definition it turns out that
	\begin{equation} 
		(T_n)^t|_{ C^\infty_{even}(S^n)}=T_n \ \text{and } (R_n)^t|_{ C^\infty_{even}(S^n)}=R_n,
	\end{equation}
	 when the space of smooth functions is embedded into the space of distribution in the usual way (see Remark~\ref{rk:smoothinDist}). Thus $(T_n)^t$ and $(R_n)^t$ can be seen as extensions of the original operators to the space of distributions, and we will still denote these extension by $T_n$ and $R_n$, see  \cite{goodeyweil}.
\end{remark}

	Observe that a convex body $K$ in $\R^{n+1}$ is a zonoid if and only if there is a non negative measure $\mu$ on $S^n$ such that:
	\be \label{eq:hT}h_K=T_n \mu.\ee
Following \cite{goodeyweil}, we define the following differential operator on $C^\infty_{even}(S^n)$:
	 \begin{equation}
	 	\square:=\Delta_{S^n}+n,
	 \end{equation}
	 where $\Delta_{S^n}$ is the Spherical Laplacian on $S^n$.
The various operators defined so far satisfy the following intertwining properties \cite[Proposition 2.1]{goodeyweil}:	
\be \label{eq:inter}\square T_n =R_n \quad \textrm{and}\quad T_n^{-1}=\square R_n^{-1}.\ee
We use now Proposition \ref{prop:zetan} to turn these operators into operators on the space of distributions on the interval, as follows.

\begin{definition}
	We define the operators $r_n, t_n$ on $\mathcal{D}'([-1,1])$ by:
	\be
	 	r_n:=\zeta_n R_n \zeta_n^{-1} \quad \textrm{and}\quad t_n:=\zeta_n T_n \zeta_n^{-1}.
	\ee
\end{definition}

\begin{remark}\label{rk:sqonI}
	Since $\square$ commutes with the action of $O(n+1)$, it preserves the space $ C^\infty_{even}(S^n)^{O(e)}$. Using the isomorphism $\zeta_n$ we consider  the operator $\zeta_n\square\zeta_n^{-1}$ that we still denote by $\square$. This is a differential operator on $C^\infty_{even}\left([-1,1]\right)$. For any $\varphi\in C^\infty_{even}([-1,1])$, it is given by 
	\begin{equation}\label{eq:Q}
		\square \varphi (z)=(1-z^2)\ \varphi^{(2)}(z)-n z  \varphi^{(1)}(z)+n \varphi(z)
	\end{equation}
In fact, let $\theta=\theta(u)$ be the angle between $e$ and $u$. By Proposition~\ref{prop:zetan}, $f:=\zeta_n^{-1}\varphi$ as a function of $\theta$ is given by $f(\theta)=\varphi(\cos(\theta))$. Equivalently $\varphi(z)=f\left(\arccos(z)\right)$. It is then enough to consider the spherical Laplacian of a function that depends only on $\theta$, which is given by $\Delta _{S^{n}} f(\theta) = (\sin\theta)^{1-n} \frac{\partial}{\partial \theta}\left((\sin\theta)^{n-1}\frac{\partial f}{\partial \theta}\right) ,$ and apply the change of variables $z=\cos(\theta)$ to obtain \eqref{eq:Q}.
\end{remark}
The following proposition gives an integral expression for $r_n$.
\begin{proposition}
	Let $\varphi\in C^\infty_{even}([-1,1])$ and let $z':=\sqrt{1-z^2}$. We have:
	\begin{equation}\label{eq:RnSpin}
		(r_n\varphi)(z)=c_n\frac{1}{(z')^{n-2}}\int_0^{z'} \varphi(w) \left((z')^2-w^2\right)^{\frac{n-3}{2}} \dd w
	\end{equation}
	where $c_n=2 \Vol_{n-2}(S^{n-2})$.
	\begin{proof}
	    Let $f:=\zeta_n^{-1}\varphi$ in such a way that $r_n\varphi=\zeta_nR_nf$. Consider an orthonormal basis $v_1,\ldots,v_n,e$ and let $u(z):=z\ e+z'\ v_n$. The unit sphere in the space $u(z)^\perp$ can be parametrized by $z'\cos\phi\ e+ z \cos\phi\ v_n+\sin\phi\ v$ where $\phi\in[0,\pi]$ and $v$ belongs to the unit sphere of $\Span\left\{v_1,\ldots,v_{n-1}\right\}$. Then, using the $O(e)$ invariance and the fact that $f$ is even, we get 
	    \begin{equation}
	        \zeta_nR_nf(z)=R_nf(u(z))=2\int_{0}^{\frac{\pi}{2}}f\left( z'\cos\phi\ e+ z \cos\phi\ v_n\right) (\sin\phi)^{n-2}\Vol_{n-2}\left(S^{n-2}\right) \dd\phi.
	    \end{equation}
	    We then note that $f\left( z'\cos\phi\ e+ z \cos\phi\ v_n\right)=\varphi(z'\cos\phi)$ and apply the change of variable $w=z'\cos\phi$ to get~\eqref{eq:RnSpin}.
	\end{proof}
\end{proposition}

	We notice the particular cases:
	\be
	\label{eq:r2}	r_2\varphi(z)	=c_2\int_0^{z'} \frac{\varphi(w)}{\sqrt{ (z')^2-w^2}} \dd w\quad \textrm{and}\quad
	r_3\varphi(z)	=\frac{c_3}{z'}\int_0^{z'} \varphi(w) \dd w.
\ee
The  equation on the left in \eqref{eq:r2} is known as Abel Equation and the second one can be inverted as follows. Recall that the operator $\mathcal{P}$ from Lemma~\ref{lem:P} is defined by $\mathcal{P}\varphi(z)=\varphi(z')$, with $z'=\sqrt{1-z^2}$.
%and its inverse and solution will be discussed in Example~\ref{eg:n=2}.
\begin{corollary}\label{cor:r3inv} The inverse of $r_3$ is given by	$r_3^{-1}=\frac{1}{c_3}\left(1+ z \frac{\dd}{\dd z}\right)\circ \mathcal{P}$. 
	\begin{proof}
		We differentiate equation~\eqref{eq:r2} to obtain $\der{z'}r_3\varphi(z)=-\frac{1}{z'}r_3\varphi(z)+\frac{c_3}{z'}\varphi(z')$. Then write $\varphi=r_3^{-1}\psi$ and change $z'$ for $z$.
	\end{proof}
\end{corollary}

The following can be deduced from~\eqref{eq:RnSpin} with a straightforward computation.
\begin{proposition}\label{prop:downtor}
	Let $\varphi\in C^\infty_{even}([-1,1])$. For any positive integer $k\leq n/2 - 1$
	\begin{equation}
		\left(\frac{1}{z'}\frac{d}{dz'}\right)^k \left((z')^{n-2}r_n\varphi(z)\right)=c_{n,k}\  (z')^{n-2-2k}r_{n-2k}\varphi(z),
	\end{equation}
	where $c_{n,k}=\frac{(n-3)!!}{(n-3-2k)!!} \frac{c_{n-2k}}{c_n}$.
\end{proposition}

Together with Corollary~\ref{cor:r3inv} this implies the following.

\begin{corollary}
	For odd $n:=2m+1$, $m\geq 1$ and for $\psi\in C^\infty_{even}([-1,1])$ the inverse of $r_n$ is given by
	\begin{equation}\label{eq:Rodd}
		\left(r_{2m+1}\right)^{-1}\psi(z)=\frac{1}{c_{2m+1,m-1}c_3}\ z\ \left(\frac{1}{z}\frac{d}{dz}\right)^m \left(z^{2m-1} \mathcal{P}\psi(z)\right)
	\end{equation}
	\begin{proof}
		Apply Proposition~\ref{prop:downtor} with $k=m-1$ to obtain 
		\begin{equation}
			\varphi(z)=\frac{1}{c_{2m+1,m-1}}r_3^{-1}\left(\frac{1}{z'}\left(\frac{1}{z'}\der{z'}\right)^{m-1}(z')^{2m-1}r_{2m+1}\varphi (z) \right).
		\end{equation}
		 Then write $\psi:=r_{2m+1}\varphi$ and apply Corollary~\ref{cor:r3inv}. It is then enough to note that, as operators, $\left(z\der{z}+1\right)\frac{1}{z}=\der{z}$.
	\end{proof}
\end{corollary}

\begin{remark}\label{rk:InverseRfordist}
	Note that in~\eqref{eq:Rodd} the coefficients of the differential operator on the right hand side are \emph{polynomial} functions of $z$. Indeed the term of lowest degree is $z\left(\frac{1}{z}\frac{d}{dz}\right)^m \left(z^{2m-1} \right)$ which is of degree zero. In other words, there exist polynomials $Q_{k,m}$ such that 
	\begin{equation}
		\left(r_{2m+1}\right)^{-1}\psi(z)=\sum_{k=0}^m Q_{k,m}(z) \left(\der{z}\right)^k \mathcal{P}\psi(z).
	\end{equation}
	Thus we can still define the operator $\sum_{k=0}^m Q_{k,m}(z) D^{k}\mathcal{P}$ on distributions. By density this is also equal to the inverse of $r_{2m+1}$.

\end{remark}
%\begin{remark}
%	For all $0\geq k \geq m$ we have $Q_{k,m}(z)=q_{k,m} z^k$. These numbers $q_{k,m}$ are given by the following recursive formula. Let $\kappa_{2l,m}:={c_{2m+1}}\int_{0}^1 s^{2l} (1-s^2)^{m-1}\dd s$. Then we have $q_{0,m}=1/\kappa_{0,2m+1}$ and
%	\begin{align}
%		q_{2l,m}		&=\frac{1}{(2l)! \kappa_{2l,m}}-\sum_{k=0}^{2l-1} \frac{q_{k,m}}{(2l-k)!}	\\
%		q_{2l+1,m}	&=-\sum_{k=0}^{2l} \frac{q_{k,m}}{(2l+1-k)!}	
%	\end{align}
%\end{remark}
Remark~\ref{rk:InverseRfordist} together with \eqref{eq:inter} and Remark~\ref{rk:sqonI} proves the following.

\begin{lemma}\label{lem:Tm-1pol}
Let $f\in C^\infty\left([-1,1]\right) $. There are polynomials $P_{k,m}$ such that
	\begin{equation}\label{eq:Pk}
		\left(t_{2m+1}\right)^{-1}\psi=\sum_{k=0}^{m+2} P_{k,m}(z) \left(\der{z}\right)^k \mathcal{P}\psi(z).
	\end{equation}
	This allows us to define the operator $\left(t_{2m+1}\right)^{-1}:=\sum_{k=0}^{m+2} P_{k,m}  D^{k} \mathcal{P}$ on $\mathcal{D}'\left([-1,1]\right)$.
\end{lemma}
Next proposition guarantees that, when applied to a constructible function, the operator on the right hand side of \eqref{eq:Pk} gives back a constructible function.
\begin{proposition}\label{propo:etac}Let $\eta:P\times S^{m+1}\times [-1,1]\to \R$ be a constructible function, and write $\eta(p,e,z)=\eta_{p,e}(z)$. Then
\be\label{eq:etac} \left((t_{2m+1})^{-1}\eta_{p,e}\right)(z):=\sum_{k=0}^{m+2}P_{k,m}(z)\left(\der{z}\right)^k \mathcal{P}\eta_{p,e}(z),\ee
is a constructible function of $(p,e,z).$
\end{proposition}
\begin{proof}Since $\eta$ is constructible, there are subanalytic functions $g_1, \ldots, g_a$ and $f_{i,1}, \ldots, f_{i, b_i}$ for $i=1, \ldots, a$ such that:
	\be \eta_{p,e}(z)=\sum_{i=1}^a\left(g_i(p,e,z)\prod_{j=1}^{b_i}\log f_{i,j}(p,e,z)\right).\ee
	Recall that the operator $\mathcal{P}$ defined in Lemma \ref{lem:P} acts as $\mathcal{P}\psi(z)=\psi(\sqrt{1-z^2})$. In particular:
	\be \mathcal{P}\eta_{p,e}(z)=\sum_{i=1}^a\left(g_i(p,e,\sqrt{1-z^2})\prod_{j=1}^{b_i}\log f_{i,j}(p,e,\sqrt{1-z^2})\right),\ee
	which is still constructible. The derivative of a constructible function is constructible (since $d\log h(z)=\frac{h'(z)}{h(z)}$) and the multiplication of a constructible function by a polynomial is constructible. Therefore the right hand side of \eqref{eq:etac} is constructible.
\end{proof}
%\commleo{Same here : we could give an explicit formula of $P_{k,m}$ in terms of $q_{k,m}$}

\section{Proof of Theorem \ref{thm:main}.}
In this section we prove Theorem \ref{thm:main}, which states that in a globally subanalytic family of convex bodies the set of zonoids is log-analytic. We formulate this as follows (the fact that we only use log-analytic functions in the definability follows from the proof).
\begin{theorem} \label{thm:zonoftameistame}
	Let $P$ be a subanalytic set and let $\left\{ K_p \st p\in P\right\}$ be a subanalytic family of convex bodies in $\R^{n+1}$. Then the set $\ZZo(P):=\set{ p\in P}{ K_p \text{ is a zonoid}}$ is definable in $\Rae$.
\end{theorem}

\begin{proof}
% What we need to prove is the following statement: let $P$ be a tame set and let $\left\{ K_p \subset \R^{n+1} \st p\in P\right\}\subset \KK^{n+1}$ be a tame family of convex bodies; then the set 
%	\begin{equation}		\ZZo(P):=\set{ p\in P}{ K_p \text{ is a zonoid}}\end{equation} is tame.
	
	Let us consider $n=2m+1$ with $m\geq1$. We will prove that the set
	\begin{equation}
		S\ZZo:=\set{(p,e)\in P\times S^n}{S_eK_p\text{ is a zonoid}}
	\end{equation}
	is definable in $\Rae$. By Lemma~\ref{lem:Lonke} we have that $\ZZo(P)=\set{p\in P}{\forall e \in S^n,\ (p,e)\in S\ZZo}$. Thus if $S\ZZo$ is definable in $\Rae$ then $\ZZo(P)$ is also definable in $\Rae$.
	
	For $p\in P$ let $h_p$ be the support function of $K_p$ and let $e\in S^n$. Then by \eqref{eq:hT} $S_eK_p$ is a zonoid if and only if $\left(T_{2m+1}\right)^{-1}S_eh_p\geq0$. 
	
	Let $\eta=\eta_{p,e}:=\zeta_nS_eh_p\in \mathcal{D}'_{even}([-1,1])$ and let $\Lambda=\Lambda_{p,e}:=t_{2m+1}^{-1}\eta\in \mathcal{D}'_{even}([-1,1])$. Then, by Remark~\ref{rk:posdist}, $S_eK_p$ is a zonoid if and only if $\Lambda\geq 0$. By Lemma~\ref{lem:SchwartzOnClosed} this is equivalent to have the following two conditions satisfied:
	 \begin{itemize}
		\item[(i)] there is $\gamma: [-1,1]\to \R$ in $W^{2,1}([-1,1])$ convex such that $D^2\gamma|_{(-1,1)}=\Lambda|_{(-1,1)}$;
		\item[(ii)] there exist $\lambda_{-1}, \lambda_{1}\geq 0$ such that $
		\Lambda-\gamma^{(2)}=\lambda_{-1}\delta_{-1}+\lambda_{1}\delta_{1}$. 
	\end{itemize}
	
	The function $S_eh_p$ is constructible as a function of $(p,e,u)$ (Corollary~\ref{cor:tamespin}) thus, by Proposition~\ref{prop:zetatame}, $\eta$ is also constructible (as a function of $(p,e,z)$). In particular there are subanalytic functions $g_1, \ldots, g_a$ and $f_{i,1}, \ldots, f_{i, b_i}$ for $i=1, \ldots, a$ such that:
	\be \eta_{p,e}(z)=\sum_{i=1}^a\left(g_i(p,e,z)\prod_{j=1}^{b_i}\log f_{i,j}(p,e,z)\right).\ee
	
	We apply now Lemma \ref{lem:tamnonsmooth} to the family of subanalytic functions $\mathcal{F}=\{g_i, f_{i,j}\}_{i=1, \ldots, a, j=1, \ldots, b_i}$ defined on the subanalytic set $S=P\times S^n$,  obtaining the corresponding partition $S=\coprod_{k=1}^s S_k$ into subanalytic sets and the subanalytic functions $a_{k,1}, \ldots, a_{k, \nu_k}:S_k\to [-1,1]$. We will show that for every $k=1, \ldots, s$ the set of points $(p,e)\in S_k$ such that $S_eK_p$ is a zonoid is definable in $\Rae$; since we are dealing with finitely many $S_k$, and since each $S_k$ is subanalytic, the result will follow.
	
	In order to simplify the notation, we omit the dependence on the subscript $k\in \{1, \ldots, s\}$ and call $S=S_k$ and $\nu=\nu_k$. For every $i=1, \ldots, \nu-1$
	we define $I_i=I_i(p,e):=]a_i,a_{i+1}[$ and $\bar{a}_i:=(a_i+a_{i+1})/2 \in I_i$ (the endpoints of this interval and its mid-point depend on $(p, e)$ in a subanalytic way) . Let also $\eta_{i}:=\eta_{p,e}(\cdot)|_{I_i}$; this is a constructible function of $(p,e,z)$, since:
	\be \eta_i(z)=\underbrace{\eta_{p,e}(z)}_{\mathrm{constructible}}\cdot\underbrace{\chi_{(a_i(p,e), a_{i+1}(p,e))}(z)}_{\mathrm{subanalytic}}.\ee
	
	  By Proposition \ref{propo:etac} $ \left(t_{2m+1}\right)^{-1} \eta_{i}$ is a constructible function of $(p,e,z)$; for every fixed $(p,e)$ it is continuous when restricted to $I_i$ and, as a distribution, by  Lemma~\ref{lem:Tm-1pol}, it equals $(\Lambda_{p, e})|_{I_i}$.
	 For $i=1,\ldots,\nu$, we define the following function on $I_i$:
	\begin{align} \label{eq:psitilde}
		\tilde{\psi}_{p,e,i}(z)&:=\int_{\bar{a}_i}^z\int_{\bar{a}_i}^s \left( \left(t_{2m+1}\right)^{-1} \eta_{i}\right)(w) \dd w   \dd s\\
		&=\int_{-1}^z\int_{-1}^s \left( \left(t_{2m+1}\right)^{-1} \eta_{p,e}(\cdot)|_{I_i(p,e)})(w)\right)\chi_{(\bar{a}_i(p,e), s)}(w)\chi_{(\bar{a}_i(p,e), z)}(s) \dd w   \dd s
	\end{align} 
	The integrand is constructible and Theorem~\ref{thm:stabbyint} implies therefore that $\tilde{\psi}_{p,e,i}$ is also constructible (as a function of $p$, $e$ and $z$).
	
	For $\alpha,\beta\in\R^{\nu}$ we let $\psi_{p,e}^{\alpha,\beta}(z)$ be the function on $[-1,1]\setminus\{a_i\}_{i=1,\ldots,\nu}$ given by 
    \begin{equation}
        \psi_{p,e}^{\alpha, \beta}(z):=\sum_{i=1}^{\nu}\left(\tilde{\psi}_{p,e,i}(z)+\alpha_i z+ \beta_i\right)\chi_{(a_i,a_{i+1})}(z).
    \end{equation}
    Condition (i) above is satisfied if and only if there exist $\alpha,\beta\in\R^{
   \nu}$ such that $\psi_{p,e}^{\alpha,\beta}$ extends to a convex function on $[-1,1]$; note that if such $\alpha, \beta$ exist, the convex function $\psi_{p,e}^{\alpha, \beta}$ is in $W^{2,1}([-1, 1])$ because its second derivative equals $\Lambda=t_{2m+1}^{-1}\eta$ (which is a distribution, thus locally, and therefore globally on $[-1,1]$, integrable).  
    
    Since $\tilde{\psi}_{p,e,i}$ is constructible, and the characteristic functions of the intervals $(a_i, a_{i+1})$ are also constructible, we can express this condition as a first order formula involving only constructible objects, therefore (i) is definable in $\Rae$.
    
    Suppose (i) is satisfied, then we can write $\Lambda_{p, e}-(\psi_{p,e}^{\alpha, \beta})^{(2)}$ as
    \be\Lambda_{p, e}-(\psi_{p,e}^{\alpha, \beta})^{(2)}=\sum_{i=0}^{m+2}\left(\lambda^{(i)}_{-1}D^{i}\delta_{-1}+\lambda^{(i)}_{1}D^{i}\delta_{1}\right),\ee 
    where the order of the distribution is bounded by $m+2$ because of Lemma \ref{lem:Tm-1pol}. Moreover the numbers $\lambda_{\pm1}^{(i)}$ are functions of $(p, e)$ which are definable in $\Rae$. Indeed by Lemma~\ref{lem:compdelt1}, their expression only involves the value of $P_{k,m}$ from \eqref{eq:Pk}, $\eta_{p,e}$ and their derivatives at the points $\pm1$. Thus condition (ii) is definable in $\Rae$ as well:
    \be \mathrm{(ii)}\iff \lambda_{\pm 1}^{(0)}\geq 0\quad \textrm{and} \quad \lambda_{\pm 1}^{(i)}=0 \quad \forall i=1, \ldots, m+2.\ee
    
    This proves that the set $S\ZZo:=\set{(p,e)\in P\times S^{2m+1}}{\Lambda_{p,e}\geq0}$	is definable in $\Rae$, as claimed.
	
	Now if $n=2m$ is even, $m\geq 1$, we consider the natural embedding 
	\begin{align}
		\iota : \R^{2m+1}		&\hookrightarrow \R^{2m+2}	\\
		(x_1,\ldots,x_{2m+1})	&\mapsto (x_1,\ldots,x_{2m+1},0),
	\end{align}
	which induces an embedding $\iota_*$ at the level of convex bodies. A convex body $K\subset \R^{2m+1}$ is a zonoid if and only if $\iota_*K\subset \R^{2m+2}$ is a zonoid. Moreover we have $h_{\iota_*K}(u)=h_K(\pi(u))$ where $\pi:\R^{2m+2}\to \R^{2m+1}$  is the projection on the first $2m+1$ coordinates. This implies that the family $\set{\iota_*K_p}{p\in P}$ is a subanalytic family of convex bodies in $\R^{2m+2}$. In particular
	\be \ZZo(P)=\set{p\in P}{K_p \text{ is a zonoid}}=\set{p\in P}{\iota_*K_p \text{ is a zonoid}}\ee
	is definable in $\Rae$ by the previous part of the proof.
\end{proof}

%\begin{example}[The case $n=2$] \label{eg:n=2}
%	In the case $S^2\subset \R^3$, there is actually no need to use the embedding in $\R^4$ as it is done in the proof. Indeed~\eqref{eq:r2} can be inverted and the solution is given by 
%	\begin{equation}
%		r_2^{-1}\psi(z)=\frac{1}{2\pi}(\mathcal{P}\psi)(0)+\frac{z}{2\pi}\int_0^z \frac{(\mathcal{P}\psi)^{(1)}(s)}{\sqrt{z^2-s^2}}\dd s.
%	\end{equation}
%	In our case $\psi=\zeta_2S_eh_p$ the derivative of which is always integrable (\cite[precisione?]{bible}). Thus $r_2^{-1}\zeta_nS_eh_p$ is a well defined (tame) function to which we can apply the operator $\square$ and check directly whether or not the result is a positive measure. 
%\end{example}
	
\bibliographystyle{acm}
\bibliography{literature}

\end{document}